\documentclass{amsart}
\usepackage{amsfonts,amscd}

\usepackage[arrow,matrix,curve,cmtip,ps]{xy}
\newtheorem{claim}{}[section]
\newtheorem{theorem}[claim]{Theorem}
\newtheorem{lemma}[claim]{Lemma}
\newtheorem{proposition}[claim]{Proposition}
\newtheorem{corollary}[claim]{Corollary}

\def\proclaim #1. #2\par{\medbreak
\noindent{\bf#1.\enspace}{\sl#2}\par\medbreak} \makeatother

\newcommand{\co}{\mathcal{Q}}
\newcommand{\lra}{\longrightarrow}
\newcommand{\M}{\mathcal{M}}

\DeclareMathOperator{\Cdb}{\mathbb{C}}

\DeclareMathOperator{\Bdb}{\mathbb{B}}
\DeclareMathOperator{\Ddb}{\mathbb{D}}

\DeclareMathOperator{\Kdb}{\mathbb{K}}

\begin{document}

\title[Extensions of operator algebras I]
{Extensions of operator algebras I}
\date{July 31, 2007}

\author{David P. Blecher}
\address{Department of Mathematics, University of Houston, Houston, TX
77204-3008} \email[David P. Blecher]{dblecher@math.uh.edu}
\author{Maureen K. Royce}
\address{Department of Mathematics, University of Houston, Houston, TX
77204-3008} \email[Maureen Royce]{royce@math.uh.edu}

\begin{abstract}
We transcribe a portion of the theory of extensions of
$C^{*}$-algebras  to general operator algebras. We also include
several new general facts about approximately unital ideals in
operator algebras and the $C^*$-algebras which they generate.
\end{abstract}

\maketitle

\let\text=\mbox

\section{Introduction}

By an operator algebra, we mean a closed, not necessarily
selfadjoint, algebra of operators on a Hilbert space.
Our purpose here is to transcribe as much as possible of the
powerful and important $C^*$-algebraic theory of extensions, to
general operator algebras, where hopefully it will also play a role.
 Although there is
no requirement on our algebras to have any kind of identity or
approximate identity, we assume for specificity 
that all {\em ideals} of operator algebras in this
paper, and therefore also the `first term' in any extension, are
have an approximate identity. This is often not the most
interesting case (for example, nontrivial ideals in the free
semigroup operator algebras which have been studied extensively
recently, will generically have no kind of identity). However, it is
the case that is closest to the rich $C^*$-algebra theory of
extensions; it is of course no restriction at all in the case of
extensions of operator algebras by $C^*$-algebras. We remark that
some of our results have variants valid for Banach algebras which we
also have not seen in the literature, which may lend further  
justification for our endeavour.
In addition to the theory of extensions, we include several new
general facts of interest.

We now describe the contents of the paper.  In Section 2 we give
several new results about ideals in operator algebras, and about
 generated $C^*$-algebras, which will be used in later sections,
 and which are independently interesting. In
 Section 3,
we describe the basic theory of extensions of operator algebras,
following in large part the diagrammatic approach of Eilers, Loring,
and Pedersen \cite{ELP}. In parts of this section, aspects which are
very similar to the $C^*$-algebra or Banach algebra case are
described quite hastily; a more thorough exposition is given in the
second author's thesis \cite{Mth} (actually much of the contents of
the paper is amplified there, together with additional results).
In Section 4, we discuss how extensions of operator algebras are
related to extensions of containing $C^*$-algebras.  In a sequel paper
we will apply the contents of this paper to notions such as
{\em semisplit} extensions, variants of the $Ext$
semigroup (or group) bivariate functor, and exactness
of operator algebras.
Some further
developments will also be contained in  \cite{Mth},
 for example nonselfadjoint variants of some other results
from \cite{ELP,Peds,Ped2}.

An operator algebra may be thought of as a closed subalgebra of a
$C^{*}$-algebra.  We refer the reader to \cite[Chapter 2]{BLM} for
the basic facts and notations which we shall need concerning
operator algebras, such as the ones below.  A few of these may also
be found in \cite{Pnbook}.
An operator algebra is
 {\em unital} if it has an
identity of norm $1$, and is  {\em approximately unital} if it has a
contractive two-sided approximate identity (cai).
All ideals are
assumed to be two-sided and closed. By a {\em morphism} we mean a
linear completely contractive homomorphism $\theta : A \to B$
between operator algebras.   If $\theta(1) = 1$ we say that $\theta$
is {\em unital}.
 If $\theta$  takes some cai for $A$ to a cai for $B$ then we say that
 $\theta$ is {\em proper}.   We write $A^1$ for Meyer's unitization
 of a nonunital operator algebra (see \cite{Mey} or \cite[Section 2.1]{BLM}).
 The multiplier algebra of $A$
will be denoted as $\M(A)$ (see \cite[Section 2.6]{BLM}),  and
$\co(A)$ will denote the corona algebra $\M(A)/A$ of $A$, which is a
unital operator algebra (it is $(0)$ if $A$ is unital). We write
$\pi_A$ for the canonical map $\M(A) \to \co(A)$.   If $A$ is an
ideal in $B$, then the canonical morphism from $B$ into $\M(A)$ will
be denoted by $\sigma$.  To say that $A$ is an {\em essential ideal}
in $B$ is to say that $\sigma$ is one-to-one. A proper morphism
$\alpha : A \to B$ extends canonically to a unital morphism
$\bar{\alpha} : \M(A) \to \M(B)$, and this induces a unital morphism
$\tilde{\alpha} : \co(A) \to \co(B)$.   Moreover $\bar{\alpha}$ is
completely isometric if and only if $\alpha$ is completely
isometric, and implies that $\tilde{\alpha}$ is completely isometric
(see Corollary \ref{tilde}).  A {\em $C^*$-cover} of an operator
algebra $B$ is a pair $({\mathcal E},j)$ consisting of a completely
isometric homomorphism $j$ from $B$ into a $C^*$-algebra ${\mathcal
E}$, such that ${\mathcal E}$ is generated as a $C^*$-algebra by
$j(B)$.  If $B$ has a cai $(e_t)$, then it follows that $(j(e_t))$
is a cai for ${\mathcal E}$.  There exists a natural ordering and
equivalence of $C^*$-covers of $B$, and a maximal and minimal
equivalence class, $C^*_{\rm max}(B)$ and $C^*_{\rm e}(B)$
respectively.  The latter is the {\em $C^*$-envelope} of $B$, and it
is a quotient of any other $C^*$-cover of $B$.  The former,
$C^*_{\rm max}(B)$,
 is characterized
by its universal property: every completely contractive homomorphism
$\pi : B \to {\mathcal D}$ into a $C^*$-algebra ${\mathcal D}$
extends to a $*$-homomorphism $C^*_{\rm max}(B) \to {\mathcal D}$.
We note that the maximal and minimal $C^*$-covers of $B$ may be
defined to be the $C^*$-algebra generated by $B$ inside the same
$C^*$-cover of the unitization $B^1$.  Although these
are only needed in the sequel paper, we will also consider the minimal
and maximal tensor products, $\otimes_{\rm min}$ (which we sometimes
write as $\otimes$) and $\otimes_{\rm max}$, of operator
algebras (see e.g.\ \cite[Chapter 6]{BLM}).

In this paper we will work with two main categories. The first is
the category {\bf OA} of all operator algebras, with morphisms the
completely contractive homomorphisms.   The second is the
subcategory {\bf AUOA} of approximately unital operator algebras,
with the same morphisms.  On one or two occasions one needs stronger
hypotheses for the smaller category (for example the pullback
construction in {\bf AUOA} needs additional restrictions in order to
reside in the same category). In any case, in both categories we
will need to assume that the first term in any {\em extension}
(defined below) is approximately unital.  Most parts of the paper
can be read twice, once for each of the above two categories.  Often
we will only state  the {\bf AUOA} case of our results; and leave
the other case to the reader. Actually, many of the results of this
paper have variants in eight categories, the other six being: the
variants of {\bf OA} and {\bf AUOA} where the morphisms are
contractive homomorphisms, and then the `up to constants' variants
of the last four categories, where in the description of the
morphisms we change the word `contractive' to `bounded', and replace
`cais' by `bais' (bounded approximate identities).  Indeed, many
of our results have obvious variants valid for (preferably,
Arens regular) Banach algebras which we have not
seen in the Banach algebras literature (which 
focuses on quite different directions, see e.g.\ \cite{BDL}).  
We will not usually take the time to state these other
cases and variants, we leave this to the
reader (see also \cite{Mth}), and for specificity restrict our
attention to the one closest to the $C^*$-algebra case. In any case,
we will use the terms morphism, subobject, quotient, etc, in the
obvious way. Thus, for example, a subobject in {\bf OA} is a closed
subalgebra, whereas in {\bf AUOA} it is an approximately unital
closed subalgebra.

We write $\Kdb, \Bdb$ for the compact and bounded operators on a
separable Hilbert space.   We say that an operator algebra $B$  is
{\em stable} if $B \cong B \otimes \Kdb$ completely isometrically
isomorphically.

\section{Ideals and $C^*$-covers}

It is of enormous importance in $C^*$-algebra theory, that a
$*$-homomorphism (or equivalently, a contractive homomorphism) on a
$C^*$-algebra has closed range, and is a complete quotient map onto
its range. For nonselfadjoint algebras, this is not at all the case.
To see this, recall that contractive Banach space maps need not have
closed range, and even if they do, they need not induce isometries
after quotienting by their kernel.  These facts, combined with the
${\mathcal U}(X)$ construction from 2.2.10--2.2.11 of \cite{BLM},
then implies the same facts for contractive or completely
contractive homomorphisms between operator algebras.  Nonetheless,
the following results in this direction will be useful:

\begin{lemma} \label{try}  Suppose that $\theta : {\mathcal E} \to
{\mathcal F}$ is a morphism between operator algebras (that is, a
morphism in {\bf OA}), which is also a complete quotient map.
Suppose that $B$ is a closed subalgebra of ${\mathcal E}$, and there
is a cai for ${\rm Ker}(\theta)$ which either (a)\ lies in $B$, or
(b)\ has a weak* limit point in $B^{\perp \perp}$. Then $\theta(B)$
is closed, and $\theta_{|B}$ is a complete quotient map onto
$\theta(B)$. Indeed, $\theta({\rm Ball}(B)) = {\rm
Ball}(\theta(B))$, and similarly at all matrix levels.   We have
$\theta(B) \cong B/{\rm Ker}(\theta_{|B}) \subset {\mathcal E}/{\rm
Ker}(\theta)$ completely isometrically isomorphically.
\end{lemma}

\begin{proof}
Hypotheses (a) and (b) are equivalent (this follows
from a seemingly deep result in \cite{Hay,BHN}).  Let $C = \theta(B),
{\mathcal D}= {\rm Ker}(\theta)$ and $A = {\rm Ker}(\theta_{|B})$,
these are ideals in ${\mathcal E}$ and $B$ respectively, and they
have a common cai. Let $p$ be a weak* limit of the cai in the second
dual; clearly $p \in B^{\perp \perp}$ and $p$ is the (central)
support projection of ${\mathcal D}$ in ${\mathcal E}^{**}$, and
also is the support projection of $A$ in $B^{**} \cong B^{\perp
\perp} \subset E^{**}$.
Writing $1$ for the identity of a unital operator algebra containing
${\mathcal E}$, we have ${\mathcal E}^{**} (1-p) \subset {\mathcal
E}^{**}$. Indeed the map $\eta \mapsto \eta (1-p) = \eta - \eta p$
is a completely contractive projection on ${\mathcal E}^{**}$, which
is a homomorphism, and its kernel is ${\mathcal D}^{\perp \perp}$.
We deduce that
$${\mathcal D}^{**} \cong {\mathcal D}^{\perp \perp} = {\mathcal
E}^{**} p, \; \; {\mathcal F}^{**} \cong {\mathcal E}^{**}/{\mathcal
D}^{\perp \perp} \cong {\mathcal E}^{**} (1-p) ,$$ and so
$${\mathcal E}^{**} = {\mathcal E}^{**} p \oplus^\infty {\mathcal E}^{**} (1-p)
\cong {\mathcal D}^{**} \oplus^\infty {\mathcal F}^{**} .$$
Similarly, $$A^{**} \cong A^{\perp \perp} = B^{\perp \perp} p , \;
\;  (B/A)^{**} \cong B^{\perp \perp}/B^{\perp \perp} p \cong
B^{\perp \perp} (1-p) .$$   The composition of the canonical
complete contractions $B/A \to {\mathcal E}/{\mathcal
 D} \to  {\mathcal E}^{**} (1-p)$, agrees with
 the composition of the canonical
complete isometries $B/A \to B^{\perp \perp} (1-p) \to {\mathcal
E}^{**} (1-p)$.
Thus the map from $B/A$ to ${\mathcal E}/{\mathcal D}$ is a complete
isometry.  Composing it with the complete isometry ${\mathcal
E}/{\mathcal D} \to {\mathcal F}$, we obtain a complete isometry
$B/A \to {\mathcal F}$. It is easy to see that this coincides with
the composition of the canonical map $\widetilde{\theta_{\vert B}} :
B/A \to C$ induced by $\theta_{\vert B}$, and the inclusion map $C
\hookrightarrow {\mathcal F}$. It follows that
$\widetilde{\theta_{\vert B}}$ is a complete isometry and has closed
range, so that $\theta_{\vert B}$ is a complete quotient map with
closed range.  The assertion that $\theta({\rm Ball}(B)) = {\rm
Ball}({\rm Ran}(\theta))$ follows from the fact that approximately
unital ideals in (we may assume, unital) operator algebras are
$M$-ideals, and hence are proximinal (see e.g.\ \cite[Section
4.8]{BLM} and \cite{HWW}). A similar assertion holds at all matrix
levels.
\end{proof}

\begin{lemma} \label{try2}  Suppose that $B$ is a closed
subalgebra of an operator algebra ${\mathcal E}$, that
${\mathcal D}, A$ are ideals in  ${\mathcal E}, B$ respectively,
with $A \subset {\mathcal D},$ and suppose that
there is a common cai for ${\mathcal D}$ and $A$. Then $B/A \subset
{\mathcal E}/{\mathcal D}$ completely isometrically isomorphically.
\end{lemma}

\begin{proof}  Just as in the proof of the last result.
\end{proof}

 {\bf Remark.}
 \ There is a `one-sided' version of the
 last results and their proofs, where e.g.\ we have right ideals and left cai.
Again we get $B/A \hookrightarrow {\mathcal E}/{\mathcal D}$
completely isometrically.  Similarly, $B/A \hookrightarrow {\mathcal E}/{\mathcal D}$
 in the Banach algebra variant.

\begin{corollary}  \label{tilde}  If $A$ is a closed subalgebra
of an operator algebra $B$, and if they have a common cai, then
$\co(A) \subset \co(B)$ completely isometrically, via the map
$\tilde{\alpha}$ described in the introduction, taking $\alpha : A
\to B$ to be the inclusion.
\end{corollary}

\begin{lemma} \label{id}  If  $A$ is a closed approximately unital ideal in
a closed subalgebra $B$ of a $C^*$-algebra ${\mathcal E}$, then the
$C^*$-subalgebra of ${\mathcal E}$ generated by $A$ is a two-sided
ideal in the $C^*$-subalgebra of ${\mathcal E}$ generated by $B$.
\end{lemma}

\begin{proof}    By 2.1.6 in
\cite{BLM}, if $(e_t)$ is a cai in $A$ then for $b \in B, a \in A$
we have $b a^* = \lim_t \, b e_t a^* \in \overline{A A^*}$.
Similarly, $b^* a, a b^*, a^* b$ lie in the $C^*$-subalgebra
 generated by $A$, from which the result is clear.
\end{proof}

\begin{lemma} \label{cenvi} If $A$ is a closed approximately unital ideal in
an operator algebra $B$, then $C^*_e(A)$ is a closed approximately
unital ideal in $C^*_e(B)$.  More specifically, the $C^*$-subalgebra
of $C^*_e(B)$ generated by $A$ is a  $C^*$-envelope of $A$.
\end{lemma}

\begin{proof}
 We can assume that $B$ is unital, since if not,
$C^*_{\rm e}(B)$ is the $C^*$-algebra generated by
$B$ in $C^*_{\rm e}(B^1)$ (see \cite[Section 4.3]{BLM}).
First suppose that  $A = Bp$ for a central projection
$p \in B$, then it is easy to see that $C^*_{\rm e}(B) p$ is a
$C^*$-envelope of $A$.  In the general case we go to the second
duals.  If $(D,j)$ is a $C^*$-envelope of $B^{**}$, then its
$C^*$-subalgebra $C$ generated by $j(B)$ is a $C^*$-envelope of $B$
by \cite[Lemma 5.3]{BHN}.  If $p$ is the support projection of $A$
in $B^{**}$, then by the first line of the proof, $D j(p)$ is a
$C^*$-envelope of $A^{**} \cong A^{\perp \perp} = B^{**} p$. Thus by
\cite[Lemma 5.3]{BHN} again, the $C^*$-subalgebra $J$ of $D j(p)$
generated by $j(A)$ is a  $C^*$-envelope of $A$.   Just as in the
proof of the last result, $J$ is clearly an ideal in $C$.
\end{proof}

{\bf Remark.}   Although we shall not use this, the following
interesting fact follows from the last result.  If $A, B$ are as in
that result, with $A, B$ both approximately unital, then $I(A)$ may
be viewed as a subalgebra of $I(B)$.  In fact there is a projection
$p \in I(B)$ with $I(A) = p I(B) p$.  To see this, apply the last
result, the fact that $I(C^*_{\rm e}(\cdot)) = I(\cdot)$, and
\cite[Theorem 6.5]{Ham2}.

\medskip

A two-sided ideal $A$ in $B$ is {\em essential} if the canonical map
$\sigma : B \to \M(A)$ is one-to-one.  We  say that the ideal is
{\em completely essential} if $\sigma$ is completely isometric.
Later we will characterize these properties in terms of the Busby
invariant. Here we give the following characterization along the
lines of \cite{KP}:

\begin{proposition}  \label{cesid}  If $A$ is a closed
approximately unital two-sided ideal in an operator algebra $B$,
then the following are equivalent:
\begin{itemize}  \item [{\rm (i)}]  $A$ is a completely essential
ideal in $B$.
\item [{\rm (ii)}]  Any complete contraction with domain $B$ is completely
isometric iff its restriction to $A$ is completely isometric.
\item [{\rm (iii)}]   There is a $C^*$-cover ${\mathcal E}$ of
$B$ such that the $C^*$-subalgebra $J$ of ${\mathcal E}$ generated
by $A$ is an essential ideal in ${\mathcal E}$.
\item [{\rm (iv)}]  Same as {\rm (iii)}, but with ${\mathcal E} =
C^*_{\rm e}(B)$.
\item [{\rm (v)}]  If $j : B \to I(B)$ is the canonical map
into the injective envelope of $B$, then $(I(B),j_{\vert A})$ is an
injective envelope of $A$. \end{itemize}
If $B$ is nonunital, these are equivalent to
\begin{itemize}
\item [{\rm (vi)}]   $A$ is a completely essential
ideal in the unitization $B^1$.  \end{itemize}
\end{proposition}

\begin{proof}  We begin by
showing that (i) is equivalent to (vi) if $B$ is not unital.
In this case, (vi) $\Rightarrow$ (i) is
trivial. Conversely, if $\sigma : B \to \M(A)$ is completely
isometric, then by Meyer's unitization theorem (see \cite[Corollary
2.1.15]{BLM}) it follows that the canonical map $B^1 \to \M(A)$ is
completely isometric, giving (vi).

(ii) $\Rightarrow$ (i) \ This follows from the fact that the
restriction of $\sigma : B \to \M(A)$ to $A$ is completely
isometric.  That (iv) $\Rightarrow$ (iii) is obvious.

(iii) $\Rightarrow$ (i) \  The canonical map ${\mathcal E} \to
\M(J)$ is a one-to-one $*$-homomorphism, and hence completely
isometric.  Thus the restriction $\rho$ to $B$ is completely
isometric. Now $J$ and $A$ have a common cai $(e_t)$. By the proof
of (2.23) in \cite{BLM}, we have
$$\Vert [b_{ij}] \Vert = \Vert [\rho(b_{ij})] \Vert = \sup_t \, \Vert [ b_{ij} e_t ] \Vert
= \Vert [\sigma(b_{ij})] \Vert , \qquad [b_{ij}] \in M_n(B).$$

(i) $\Rightarrow$ (iv) \  We are supposing
$\sigma : B \to \M(A)$ is completely
isometric.  View
$C^*_{\rm e}(A) \subset C^*_{\rm e}(B)$ as in
  Lemma \ref{cenvi},  and consider the
canonical $*$-homomorphism $\sigma' : C^*_{\rm e}(B) \to \M(C^*_{\rm e}(A))$.
Since $C^*_{\rm e}(A)$ and $A$ have a common cai $(e_t)$,
the last centered equation in the last paragraph, shows in the current
setting that the restriction $\rho$ of $\sigma'$ to $B$ is completely
isometric.  By
the `essential property' of the $C^*$-envelope (see e.g.\ 4.3.6 in
\cite{BLM}), $\sigma'$  is completely
isometric.

(v) $\Rightarrow$ (ii) \ Given a complete contraction $T : B \to
B(H)$ whose restriction to $A$ is completely isometric, extend $T$
to a complete contraction $\hat{T} : I(B) = I(A) \to B(H)$.
 By the `essential property' of $I(A)$ (see e.g.\ \cite[Section 4.2]{BLM}),
$\hat{T}$ is completely
isometric, and hence also $T$.

(iv) $\Rightarrow$ (v) \  
We may assume that $B$ is approximately unital, since in the 
contrary case one may appeal to
(vi), and also use the fact that $C^*_{\rm e}(B)$ is the
$C^*$-algebra generated by $B$ in $C^*_{\rm e}(B^1)$ (see
\cite[Section 4.3]{BLM}).
Then this follows from  the $C^*$-algebraic case of
(i) $\Rightarrow$ (v) from \cite{KP}, together with the fact that
$I(A) = I(C^*_{\rm e}(A))$ (and similarly for $B$).
\end{proof}

{\bf Remark.}  The proof shows that the conditions are also
equivalent to $B$ being  $A$-$A$-essential (resp.\
$A$-$\Cdb$-essential)  in the sense of \cite{KP}.

\begin{lemma}  \label{max}  If $A$ is approximately unital, and
is an ideal in an operator algebra $B$,
define ${\mathcal D}$
to be the $C^*$-subalgebra of  $C^*_{\rm max}(B)$ generated by $A$.  Then  ${\mathcal D}$ (resp.\ $C^*_{\rm
max}(B)/{\mathcal D}$) is a maximal $C^*$-cover of $A$ (resp.\ of
$B/A$).  That is, $C^*_{\rm max}(A) = {\mathcal D}$ and $C^*_{\rm
max}(B/A) = C^*_{\rm max}(B)/{\mathcal D}$.
\end{lemma}

\begin{proof}  We first prove that ${\mathcal D}$ above has
the universal property characterizing $C^*_{\rm max}(A)$. If $\pi :
A \to B(H)$ is a nondegenerate completely contractive homomorphism,
then by 2.6.13 of \cite{BLM}, $\pi$ extends to a completely
contractive homomorphism $B^1 \to B(H)$, and hence to a
$*$-homomorphism from $C^*_{\rm max}(B^1) \to B(H)$. The restriction
of the latter to ${\mathcal D}$ is a nondegenerate $*$-homomorphism
extending $\pi$. Thus ${\mathcal D} = C^*_{\rm max}(A)$.  By Lemma
\ref{try2}, $B/A \subset C^*_{\rm max}(B)/{\mathcal D}$ completely
isometrically, and it is easy to see that $C = B/A$ generates the latter
$C^*$-algebra, and so this $C^*$-algebra is a $C^*$-cover of $C$.
The fact that it is the maximal one follows by showing that it has
the universal property characterizing $C^*_{\rm max}(C)$: a
nondegenerate completely contractive homomorphism $\pi : C \to B(H)$
induces a homomorphism $B \to B(H)$ which annihilates $A$, which in
turn extends to a $*$-homomorphism $C^*_{\rm max}(B) \to B(H)$ which
annihilates ${\mathcal D}$.  This induces a $*$-homomorphism on
$C^*_{\rm max}(B)/{\mathcal D}$, and it is easy to check that this `extends' $\pi$.
\end{proof}

The following results, which are needed in the sequel paper,
 use the language of operator algebra tensor
products (see e.g.\ \cite[Section 6.1]{BLM}):

\begin{lemma} \label{maxp}  If $B$ is any $C^*$-algebra,
and if $A$ is any approximately unital operator algebra, then
$C^*_{\rm max}(B \otimes_{\rm max} A) = B \otimes_{\rm max} C^*_{\rm
max}(A)$.  If $B$ is in addition a nuclear $C^*$-algebra,
then $C^*_{\rm max}(B \otimes_{\rm min} A) = B
\otimes_{\rm min} C^*_{\rm max}(A)$.
\end{lemma}

\begin{proof}
By (6.3) in \cite{BLM}, we have $B \otimes_{\rm max} A  \subset B
\otimes_{\rm max} C^*_{\rm max}(A)$.  Clearly $B \otimes A$
generates the latter $C^*$-algebra.
 We show that $B \otimes_{\rm max} C^*_{\rm max}(A)$ has the
universal property of $C^*_{\rm max}(B \otimes_{\rm max} A)$. Let
$\theta : B \otimes_{\rm max} A \to B(H)$ be a completely
contractive homomorphism.  By \cite[Corollary 6.1.7]{BLM}, there are
two completely contractive homomorphisms $\pi : B \to B(H)$ and
$\rho : A \to B(H)$ with commuting ranges such that $\theta(b
\otimes a) = \pi(b) \rho(a)$.   Now $\pi$ is forced to be a
$*$-homomorphism by \cite[Proposition 1.2.4]{BLM}, and hence the
range of the canonical extension $\tilde{\rho}$ of $\rho$ to
$C^*_{\rm max}(A)$ commutes with $\pi(B)$.  Hence we obtain a
$*$-homomorphism $\tilde{\theta} : B \otimes_{\rm max} C^*_{\rm
max}(A) \to B(H)$ with
$$\tilde{\theta}(b \otimes a) = \pi(b) \tilde{\rho}(a) =  \pi(b)
\rho(a) = \theta(b \otimes a) , \qquad a \in A, b \in B ,$$  proving
the result.
 \end{proof}

\begin{lemma} \label{minpre}  If $A, B$ are
 approximately unital operator algebras then $B \otimes_{\rm min} A$
 is a completely  essential ideal in $B^1 \otimes_{\rm min} A^1$.
 Here $A^1$ is the unitization, set equal to $A$ if $A$ is already
 unital, and similarly for $B^1$.
\end{lemma}

\begin{proof}
Let $\sigma : B^1 \otimes_{\rm min} A^1 \to \M(B \otimes_{\rm min}
A)$ be the canonical morphism.  Assume $A$ and $B$ are
nondegenerately represented on Hilbert spaces $K$ and $H$
respectively.  Then $B^1 \otimes_{\rm min} A^1$ may be regarded as a
unital subalgebra of $B(H \otimes K)$.  For $u \in B^1 \otimes_{\rm
min} A^1$ and $\zeta \in {\rm Ball}(H \otimes K)$, we have $$\Vert
\sigma(u) \Vert \geq \Vert u (f_s \otimes e_s) \Vert \geq \Vert u
(f_s \otimes e_s) \zeta \Vert .$$ Taking a limit gives $\Vert
\sigma(u) \Vert \geq \Vert u  \zeta \Vert$, so that $\Vert \sigma(u)
\Vert \geq \Vert u \Vert$. So $\sigma$ is an isometry, and similarly
it is a complete isometry.
\end{proof}

\begin{theorem} \label{minp}  If  $A$ and $B$ are two operator systems,  or
two approximately unital operator
algebra, then $C^*_{\rm e}(B \otimes_{\rm min} A) = C^*_{\rm e}(B)
 \otimes_{\rm
min} C^*_{\rm e}(A)$.
\end{theorem}

\begin{proof}  First assume that $A, B$ are unital.  In this
case, one may assume below that $A, B$ are operator systems if one
likes, by replacing $A$ by $A + A^*$, and similarly for $B$. Then
the result is proved in \cite[Theorem 6.8]{Ham2}. 
We include a more
modern proof for the readers convenience. Let $\Phi : A \to B(H)$ be
a completely isometric unital boundary representation in the sense of
 \cite{DM} (this paper is simplified in \cite{Arvnt}, where
these maps are said to have the {\em unique extension property}).
Then $\Phi$ extends to a unital
$*$-monomorphism from $C^*_{\rm e}(A)$ into $B(H)$, by definition of
a boundary representation, and the `essential' property of $C^*_{\rm
e}(A)$.  So we may
identify $A$ as a subspace of $B(H)$ such that the $C^*$-algebra it
generates inside $B(H)$ is $C^*_{\rm e}(A)$.  Similarly, we may
assume that $B \subset C^*_{\rm e}(B) \subset B(K)$ for a Hilbert
space $K$, with the inclusion map being a boundary representation.
We may view $B \otimes_{\rm min} A$ as a subspace of $B(K \otimes
H)$, and $C^*_{\rm e}(B) \otimes_{\rm min} C^*_{\rm e}(A)$ is the
$C^*$-subalgebra it generates, in $B(K \otimes H)$. The injective
envelope $I(B \otimes_{\rm min} A)$ may be viewed as a subspace of
$B(K \otimes H)$, and so the canonical map $\pi : C^*_{\rm e}(B)
 \otimes_{\rm min} C^*_{\rm
e}(A) \to C^*_{\rm e}(B \otimes_{\rm min} A)$ may be viewed as a
completely positive unital map into $B(K \otimes H)$. It suffices to
show that $\pi$ is one-to-one.    Let $\theta(y) = \pi(1 \otimes
y)$, a completely positive unital map $C^*_{\rm e}(A) \to B(K
\otimes H)$ extending $I_K \otimes I_A$.   The map $y \to I_K
\otimes y$ on $C^*_{\rm e}(A)$
 is a boundary representation too,
since any `multiple' of a boundary representation is easily seen to
be a boundary representation (using \cite[Proposition 4.1.12]{BLM}
if necessary). It follows that $\theta(y) = I_K \otimes y$ for all
$y \in C^*_{\rm e}(A)$.   Similarly, $\pi(x \otimes 1) = x \otimes
I_H$ for all $x \in C^*_{\rm e}(B)$.  Because of the latter, it
follows by 1.3.12 in \cite{BLM} that $\pi(x \otimes y) = (x \otimes
1) \pi(1 \otimes y)$ for all $x \in C^*_{\rm e}(B),
 y \in C^*_{\rm e}(A)$.  Thus  $\pi$ is the `identity
map', and is completely isometric.

Next, suppose that
$A, B$ have cais $(e_t), (f_s)$ respectively. Let $J$ be a boundary
ideal (see e.g.\ \cite{SOC} and p.\ 99 in \cite{BLM}) for $B
\otimes_{\rm min} A$ in $C^*_{\rm e}(B) \otimes_{\rm min} C^*_{\rm
e}(A)$. Then $J$ is also an ideal in $C^*_{\rm e}(B)^1 \otimes_{\rm
min} C^*_{\rm e}(A)^1$. Let $\theta : B^1 \otimes_{\rm min} A^1 \to
(C^*_{\rm e}(B)^1 \otimes_{\rm min} C^*_{\rm e}(A)^1)/J$ be the
canonical completely contractive morphism factoring through
$C^*_{\rm e}(B)^1 \otimes_{\rm min} C^*_{\rm e}(A)^1$. The
restriction of $\theta$ to $B \otimes_{\rm min} A$ is completely
isometric, being the composition of the canonical morphism $B
\otimes_{\rm min} A \to (C^*_{\rm e}(B) \otimes_{\rm min} C^*_{\rm
e}(A))/J$, and the `inclusion' $(C^*_{\rm e}(B) \otimes_{\rm min}
C^*_{\rm e}(A))/J \to (C^*_{\rm e}(B)^1 \otimes_{\rm min} C^*_{\rm
e}(A)^1)/J$. It follows from Proposition \ref{cesid} and Lemma
\ref{minpre} that $\theta$ is completely isometric. Hence $J = (0)$,
proving our result. \end{proof}

If we define the cone and suspension of a nonselfadjoint operator
algebra just as one does in the $C^*$-literature (we will not take
the time to review this; but remark there is a unitized and
nonunitized version of these constructions, both of which work for
our purposes in the sequel paper),
it follows from the last two results that:

\begin{corollary} \label{cos}  The cone and suspension operations both
commute with both of $C^*_{\rm e}$ and $C^*_{\rm max}$, at least for
approximately unital operator algebras.
\end{corollary}

\section{Theory of extensions}

\subsection{The pullback}  Given three objects $A, B, C$ in a
 category, and morphisms $\alpha : A \to C, \beta : B \to C$, the
{\em pullback} $A \oplus_C B$ of $A$ and $B$ (along $\alpha$ and
$\beta$), is the object which, together with two fixed morphisms
$\gamma, \delta$ from this object to $A$ and $B$ respectively,
satisfies the universal property/diagram
$$
\xymatrix{ & A \ar[dl]_\alpha & & & & \\
C & & A \oplus_C B \ar[ul]^\gamma \ar[dl]_\delta & \ar[ull]_\mu
\ar[dll]^\nu D
\ar@{-->}[l]^\pi \\
& B \ar[ul]_\beta & & & \\
    } $$
That is, for any object $D$ and morphisms $\mu : D \to A, \nu : D
\to B$ with $\alpha \mu = \beta \nu$, there exists a unique morphism
$\pi : D \to A \oplus_C B$ such that the diagram commutes.
 By an obvious variant of the
usual argument, the pullback is unique up to the appropriate completely
isometric isomorphism.  Indeed, in  our
setting, concretely
$$A \oplus_C B = \{ (a,b) \in A \oplus^\infty B : \alpha(a) =
\beta(b) \} ,$$ and the morphisms $\gamma, \delta$ are just the
projections onto the two coordinates of the set in the last
displayed equation.   The map $\pi$ takes $d \in D$ to
$(\mu(d),\nu(d))$, of course.

It is easy to see that the pullback is closed (complete).  In the
category of approximately unital operator algebras, it is not true
in general that the pullback is approximately unital. However, in
Subsection \ref{diag1} we will give a condition under which it
always will be.

The {\em pushforward} construction in the category {\bf AUOA} works just
as in the $C^*$-algebra case.  Many results in the rest of this
paper can be phrased in terms of pushforwards, just as in
\cite{ELP,Peds} etc., and we leave this to the reader (see
\cite{Mth} for more details).

\subsection{Extensions} \label{ext}

If $A, C$ are nontrivial operator algebras, with $A$ approximately
unital,  then an {\em extension of $C$ by $A$} is an exact sequence
$$0 \lra A\overset{\alpha}\lra B\overset{\beta}\lra C \lra 0$$
where $B$ is an operator algebra, and $\alpha, \beta$ are completely
contractive homomorphisms, with $\alpha$ completely isometric, and
$\beta$ a complete quotient map.  Some of our results will
have variants valid with a weaker assumption on $A$ (cf.\ 
the annihilator ideal assumptions in \cite[Theorem 1.2.11]{Palm}), 
but for specificity we will not consider this here.
If we are working in the category {\bf AUOA} we will also want to
assume that $B, C$ are approximately unital too, of course.
Actually,
it usually causes no trouble if we assume that $A \subset B$, and
$\alpha$ is the inclusion map. Thus in future, we will sometimes
silently be assuming  this.

 The canonical example of an extension
is the {\em corona extension}
$$ \xymatrix{
   0  \ar[r] &   A   \ar[r]  &  \M(A)   \ar[r] &
   {\mathcal Q}(A)  \ar[r] & 0 \\
    } $$
This is the largest {\em essential extension} with first term $A$
(see \cite{Peds,Mth} for more details).  It is also what we call
a {\em unital extension}, namely an extension whose middle term
is unital.

\begin{proposition} \label{aau}  Given an extension
$$0 \lra A\overset{\alpha}\lra B\overset{\beta}\lra C  \lra 0$$
in the above sense,  $C$ is approximately unital iff
$B$ is approximately
unital.
\end{proposition}

\begin{proof}  The one direction is obvious.  For the other,
we may suppose without loss of generality that $A \subset B \subset
D$, where $D$ is a unital operator algebra with identity which is
the $1$ appearing below, and $C = B/A$. Suppose that $p$ is the
(central) support projection for $A$ in $B^{**} = B^{\perp \perp}$.
As in the proof of Lemma \ref{try} we have $$A^{\perp \perp} =
B^{**} p \, , \; \; \;
B^{**}/A^{\perp
\perp} \cong B^{**} (1-p) \cong C^{**} .$$ Thus $B^{**} (1-p)$ is
unital, and hence so is $B^{**} = B^{**} p \oplus^\infty B^{**}
(1-p)$.  Thus $B$ is approximately unital (see \cite[Proposition
2.5.8]{BLM}).
\end{proof}

{\bf Remarks.} 1) \
Note that if $B$ is a $C^*$-algebra then $A, C$ are also
$C^*$-algebras.  Conversely, if $A, C$ are $C^*$-algebras, then $B$
is a $C^*$-algebra.  To see this note that $B^{**} \cong A^{**}
\oplus^\infty C^{**}$, and one may appeal to \cite[Lemma
7.1.6]{BLM}.

2) \ An extension of the complex numbers by $A$, with the middle
algebra $B$ unital, is exactly the same thing as a unitization of
$A$.

\subsection{Morphisms between extensions}  Given two
extensions of $C$ by $A$, it is of interest to find a dotted arrow
(namely, find a completely contractive homomorphism) making the
following diagram commutative:
$$
\xymatrix{ 0 \ar[r] & A  \ar[r]  \ar@{=}[d]  & B_1
    \ar[r] \ar@{-->}[d] &
    C \ar@{=}[d]   \ar[r] & 0\\
   0  \ar[r] &   A   \ar[r] &  B_2  \ar[r] &
    C \ar[r] & 0 \\
    } $$
If such a morphism exists, then it is unique (as we shall say again
later in a more general setting). We may regard the existence of
this morphism as giving a partial ordering on the set of extensions
of $C$ by $A$. If the dotted arrow is also a surjective complete
isometry then we say that the two extensions are {\em strongly
isomorphic}.    We write {\bf Ext}$(C,A)$ for the set of equivalence
classes with respect to strong isomorphism of extensions of $C$ by
$A$.  This is easily seen to be a `bi-functor'. Indeed any morphism
$\theta : C_1 \to C_2$ gives a map ${\bf Ext}(C_2,A) \to {\bf
Ext}(C_1,A)$ by what we call `Diagram I' below.  We remark that
Diagrams I, II, III, IV of \cite{ELP} work in our setting
essentially just as in that paper. Any proper morphism $A_1
 \to A_2$ gives  a map ${\bf Ext}(C,A_1) \to {\bf Ext}(C,A_2)$,
 by `Diagram III'.  If the
morphism is not proper then the functoriality is much deeper (see
the sequel to this paper).

As usual there is a notion of {\em split extension}, namely that
there exists a morphism, $\gamma : C \to B$ such that $\beta
\gamma=I_{C}$.  These include, but usually do not coincide with, the
extensions strongly isomorphic to the trivial extension (the  one
with
 $B = A \oplus^\infty C$, and $\alpha, \beta$ the obvious
 maps).  See \cite{Bus}.  A split extension is {\em strongly
 unital} if $\gamma$ above is also unital.
 Note that the second dual of any
 extension is, by the proof of Lemma \ref{try}, an extension
 which is strongly isomorphic to
 the trivial
 extension.  Along those lines we remark that a short exact sequence
 is an extension in the sense of our paper iff its second dual
 is an extension.

More generally, given two extensions in {\bf Ext}$(C_1,A_1)$ and
{\bf Ext}$(C_2,A_2)$, a morphism from the first to the second is a
commutative diagram:

$$
\xymatrix{ 0 \ar[r] & A_1  \ar[r]  \ar[d]  & B_1
    \ar[r] \ar[d] &
    C_1 \ar[d]   \ar[r] & 0\\
   0  \ar[r] &   A_2   \ar[r] &  B_2  \ar[r] &
    C_2 \ar[r] & 0 \\
    } $$
    We will see in Theorem \ref{bumo} that if the
    left vertical arrow is
    proper, then the vertical arrow
    in the middle is uniquely
    determined by the other two vertical arrows, and we also list
   there a criterion for its existence.
Two extensions of $C$ by $A$ are {\em weakly equivalent} if there
exists a morphism from one onto the other, in the latter sense, with
all three vertical arrow surjective complete isometries.  There are
other notions of equivalence which we consider in the sequel to
this paper and \cite{Mth}.   

We state a variant of the `five lemma' from algebra:

\begin{lemma} \label{five}  Given a  morphism  between extensions
as in the above diagram: if the outer vertical arrows are complete
isometries (resp.\ complete quotient maps), then so is the middle
vertical arrow.
\end{lemma}

\begin{proof}
 Taking second duals of all algebras and morphisms in the diagram,
we may assume that both rows in the diagram are trivial extensions.
It then is easy algebra to see that the middle map is of the form
$\rho((a,c)) = (\mu(a) + \delta(c),\nu(c))$, for a morphism $\delta$,
where
$\mu, \nu$ are the outer vertical arrows.   Since $\rho(1,0) = \rho(1,0) \rho(1,c)$,
it follows that the projection $\mu(1)$ is orthogonal to the range
of $ \delta$.  Thus if $\mu, \nu$ are complete quotient maps,
then $\delta = 0$, and $\rho$ is then easily seen a complete quotient map.
We also have $\Vert \rho((a,c)) \Vert = \max \{ \Vert \mu(a) \Vert , \Vert \delta(c) \Vert , \Vert
\nu(c) \Vert \}$, and similarly for matrices, from which the complete isometry case
follows.
\end{proof}

\subsection{Diagram I} \label{diag1} \ The next tool we mention is what is
called {\em Diagram I} in \cite{ELP,Peds}:
$$
\xymatrix{
    0 \ar[r] & \circ  \ar@{-->}[r]  \ar@{-->}[d]  & \circ
    \ar@{-->}[r] \ar@{-->}[d] &
    C' \ar[d]^{\gamma}   \ar[r] & 0\\
   0  \ar[r] &   A   \ar[r]^{\alpha}  &  B   \ar[r]^{\beta} &
    C \ar[r] & 0 \\
    } $$
That is, given an extension of operator algebras (the bottom row),
and a completely contractive homomorphism $\gamma$ from an operator
algebra $C'$ as shown ($C$ is approximately unital of course, if we
are working in {\bf AUOA}), then we can complete the first row to an
extension of operator algebras, and find vertical morphisms so that
the diagram commutes. In fact, one completion of the diagram is as
follows
$$
\xymatrix{
    0 \ar[r] & A  \ar@{-->}[r]^{\tilde{\alpha}}  \ar@{=}[d]  & B
    \oplus_C C'
    \ar@{-->}[r]^{q_2} \ar@{-->}[d]^{q_1} &
    C' \ar[d]^{\gamma}   \ar[r] & 0\\
   0  \ar[r] &   A   \ar[r]^{\alpha}  &  B   \ar[r]^{\beta} &
    C \ar[r] & 0 \\
    } $$
Here $B \oplus_C C' = \{ (b,c') \in B \oplus^\infty C' : \beta(b) =
\gamma(c') \}$ is a pullback.
The maps $q_1, q_2$ in the diagram are the
canonical projection onto $B$ and $C'$ respectively.
It is easy to see that $q_2$ is a complete quotient map, and further
details are just as in \cite[Theorem 1.2.10]{Palm}.  
Note that it follows
from Proposition \ref{aau} that the pullback $B \oplus_C C'$ is
approximately unital if $C'$ is. 

This `completion' of Diagram I is the universal
one. That is, given any other extension constituting the top row of
a commuting Diagram I, this extension factors through the one in the
last paragraph, just as in \cite{ELP}.

As we remarked earlier,  Diagrams II, III, IV of \cite{ELP} also
work in our setting just as in that paper, and we will use these
tools without comment.  Similarly for their notion of {\em corona
extendibility}, which we shall not study here.

\subsection{The Busby invariant}

 Given any extension $$E \, : \; \; \; 0 \to A \overset{\alpha}\lra B
\overset{\beta}\lra C \to 0 $$
 of $C$ by $A$, there is a morphism $\tau : C \to {\mathcal Q}(A)$
 defined by $\tau(c) = \sigma(b)  + A$, where $\sigma : B \to \M(A)$
 is the canonical map (namely, $\sigma(b)(a) = \alpha^{-1}(b
 \alpha(a))$, for $a \in A, b \in B$), and $\beta(b) = c$.
 We call $\tau$ the {\em Busby
 invariant} of the extension $E$.  The Banach algebra variant 
may be found in \cite[Theorem 1.2.11]{Palm}, which we now explain
briefly in our context.
 If $\tilde{\beta}
 : B/\alpha(A) \to C$ is the surjective complete isometry induced by
  the complete quotient map $\beta$, and if $\tilde{\sigma} :
  B/\alpha(A) \to \M(A)/A$ is the complete contraction induced
  from $\sigma : B \to \M(A)$, then $\tau = \tilde{\sigma} \circ
  \tilde{\beta}^{-1}$.
 This shows that $\tau$ is well defined and completely
 contractive.

 Given two operator algebras $A, C$, with at least $A$
 approximately unital, and
 a completely contractive homomorphism $\gamma : C \to
{\mathcal Q}(A)$, the Diagram I tool above gives an
extension of $C$ by $A$:
$$
\xymatrix{
    0 \ar[r] & A  \ar@{-->}[r]^{\tilde{\iota}}
     \ar@{=}[d]  & {\rm PB}    \ar@{-->}[r]^{q_2} \ar@{-->}[d]^{q_1} &
    C \ar[d]^{\gamma}   \ar[r] & 0\\
   0  \ar[r] &   A   \ar[r]^{\iota}  &  \M(A)   \ar[r]^{\pi_A} &
   {\mathcal Q}(A)  \ar[r] & 0 \\
    } $$
    Here PB is the pullback $\M(A) \oplus_{{\mathcal Q}(A)} C$.
    We call this the {\em pullback extension constructed from $\gamma$}.
As usual, the Busby invariant of the latter extension is exactly
$\gamma$.  Conversely,   
any extension
 $$E \, : \; \;  0\lra A\overset{\alpha}\lra B\overset{\beta}\lra C\lra 0 $$
 is strongly
 isomorphic to the
 pullback
extension constructed from $\tau$ as in the last paragraph, taking
$\tau$ to be the Busby invariant of the  extension $E$. That is,
there is a completely isometric surjective morphism $\varphi$ making
the following commute:
$$ \xymatrix{
    0 \ar[r] & A  \ar[r]^\alpha \ar@{=}[d] &  B   \ar[r]^\beta
    \ar[d]^\varphi &
      C \ar@{=}[d]   \ar[r] & 0\\
   0 \ar[r] & A  \ar[r]^{\tilde{\iota}} &
     PB    \ar[r]^{q_{2}} &    C\ar[r] & 0 \\
  } $$
where PB is the pullback $\M(A) \oplus_{{\mathcal Q}(A)} C$ along
$\pi_A$ and  $\tau$.
Indeed, by definition of the pullback there is a
canonical completely contractive morphism $\varphi : B \to \M(A)
\oplus_{{\mathcal Q}(A)} C$, given by
$\varphi(b) = (\sigma(b),\beta(b))$, for $b \in B$. That  
 $\varphi$ is completely isometric and surjective,
 follows from Lemma \ref{five}.

As usual, the {\em trivial extension} (the  one with
$B = A \oplus^\infty C$), has Busby invariant
the zero map, and we also have (see \cite[Theorem 1.2.11]{Palm}):

\begin{theorem} \label{bus}  There is a bijection between {\bf
Ext}$(C,A)$ and {\rm Mor}$(C,{\mathcal Q}(A))$, the space of
completely contractive homomorphisms $\tau : C \to {\mathcal Q}(A)$.
There is a (non-bijective) correspondence between the equivalence
classes of split extensions and {\rm Mor}$(C,{\mathcal M}(A))$. In
fact, an extension is split precisely when its Busby invariant
equals $\pi_A \circ \eta$ for some $\eta \in {\rm Mor}(C,{\mathcal
M}(A))$.  If $C$ is unital the bijection above restricts to a
bijection between the unital extensions (that is, those with middle
term a unital algebra), and unital morphisms (those taking $1$ to
$1$).
 \end{theorem}

Thus we often refer to an element of ${\rm Mor}(C,\co(A))$ as an
extension of $C$ by $A$.

 Recall
that given two extensions in {\bf Ext}$(C_1,A_1)$ and {\bf
Ext}$(C_2,A_2)$, a morphism from the first to the second is a
commutative diagram:
$$
\xymatrix{ 0 \ar[r] & A_1  \ar[r]  \ar[d]^\mu  & B_1
    \ar[r] \ar[d]^\rho &
    C_1 \ar[d]^\nu   \ar[r] & 0\\
   0  \ar[r] &   A_2   \ar[r] &  B_2  \ar[r] &
    C_2 \ar[r] & 0 \\
    } $$
    If $\mu$ is a proper morphism, then
    $\mu$ extends to a map $\M(A_1) \to \M(A_2)$, which in
    turn induces a map $\tilde{\mu} : \co(A_1) \to \co(A_2)$.
    This notation is used in the following:

    \begin{theorem} \label{bumo}  Given two extensions as above,
    and morphisms $\mu : A_1 \to A_2$ and $\nu : C_1 \to C_2$,
    with $\mu$ proper, then there exists a (necessarily unique)
    morphism $\rho : B_1 \to
    B_2$ so that the diagram above commutes if and only if
we have $\tilde{\mu} \circ \tau_1 = \tau_2 \circ \nu$, in the
notation above.  Here $\tau_1, \tau_2$ are the Busby invariants of
the two
    sequences.
\end{theorem}

\begin{proof}  See \cite[Lemma 2.1 and Theorem 2.2]{ELP} for details of the proof.
\end{proof}

{\bf Remark.}  The result \cite[Theorem 2.4]{ELP}, which is
technically important in that paper, also easily carries over to
nonselfadjoint algebras, as well as many of its consequences.  See
\cite{Mth} for details.

\medskip

An extension will be called {\em essential} if $\alpha(A)$ is an
essential ideal in $B$ in the sense that the canonical  map $\sigma
: B \to \M(A)$ is one-to-one; this turns out to be equivalent to the
associated Busby invariant
 $\tau$ being  one-to-one.  We  say that an extension is
{\em  completely essential} if $\sigma$ is completely isometric.

\begin{lemma} \label{cess} An extension is completely essential iff
the associated Busby invariant is completely isometric.
\end{lemma}

\begin{proof}
 The definition of the Busby invariant $\tau$ gives
a morphism of extensions:
$$
\xymatrix{ 0 \ar[r] & A  \ar[r]^\alpha  \ar@{=}[d]  & B
    \ar[r]^\beta \ar[d]^\sigma &
    C \ar[d]^\tau   \ar[r] & 0\\
   0  \ar[r] &   A   \ar[r] &  \M(A)  \ar[r]^\pi &
    \co(A) \ar[r] & 0 \\
    } $$
 By Lemma \ref{five}, if $\tau$ is  completely isometric then so
 is $\sigma$.  Conversely, if  $\sigma$ is  completely isometric,
 then it is easy to see using Lemma \ref{try2}
that so is the map $\tilde{\sigma}$ that
 we used in the definition of the Busby invariant, and hence
 $\tau$ is completely isometric by the formula $\tau =
\tilde{\sigma} \circ \tilde{\beta}^{-1}$ given in that place.
 \end{proof}

\subsection{Subextensions}

A {\em subextension} of an extension $$0 \to {\mathcal
 D} \overset{\alpha}\lra
 {\mathcal E} \overset{\beta}\lra {\mathcal F} \to 0 ,$$
consists of closed subalgebras $A, B, C$ of ${\mathcal
 D},
 {\mathcal E},$ and $ {\mathcal F}$ respectively,
 such that we have an extension
 $$0 \lra A \overset{\alpha_{| A}}\lra B \overset{\beta_{| B}}\lra C
 \lra 0 .$$
Of course this forces $C = \beta(B)$, and $A =
 \alpha^{-1}(B)$.  To see this, note that clearly $A \subset
 \alpha^{-1}(B)$.  On the other hand, the canonical isomorphism
 $B/\alpha(A) \to C$ is the composition of the canonical map $B/(B \cap
 \alpha({\mathcal D})) \to C$, and the canonical map $B/\alpha(A) \to
 B/(B \cap
 \alpha({\mathcal D}))$.  Thus the latter map is one-to-one, which
 implies that $B \cap
 \alpha({\mathcal D}) = \alpha(A)$.  Hence $A = \alpha^{-1}(B)$.

\begin{proposition}  \label{etc}   Given an extension $$E \; \; : \; \;
0 \to {\mathcal
 D} \overset{\alpha}\lra {\mathcal E} \overset{\beta}\lra {\mathcal F} \to 0 ,$$
 and a
 nontrivial subobject  $B$ of
 ${\mathcal E}$, then there exists a subextension of $E$ with `middle term' $B$
 if and only if
 $\alpha({\mathcal
 D}) \cap B$ is approximately unital, and $\beta|_{B}$ is a complete
 quotient map.  If $B$ contains a cai for $\alpha({\mathcal D})$
 then the last condition (that $\beta|_{B}$ is a complete
 quotient map) is automatically true.
 \end{proposition}

\begin{proof}
One direction is obvious from the discussion above.
For the other, if
$\alpha({\mathcal
 D}) \cap B$ is  approximately unital, let $C = \overline{\beta(B)}$
 and $A = \alpha^{-1}(B) =
 \alpha^{-1}(B \cap
 \alpha({\mathcal D}))$.  If $\beta(b) = 0$ then
 $b = \alpha(d)$ for some $d \in {\mathcal
 D}$.  Clearly, $d \in A$.  Thus Ker$(\beta_{|B}) = \alpha(A)$,
 and so $B$ is an extension of $C$ by $A$ if $\beta|_{B}$ is a
complete
 quotient map.  The last assertion follows from Lemma \ref{try}.
  \end{proof}

 We just saw that the middle algebra
$B$ in a subextension as above determines $A$ and $C$. We now
discuss how $A$ or $C$ determines the others. An approximately
unital closed subalgebra $C$ of ${\mathcal F}$ gives a subextension
$$0 \lra {\mathcal D} \lra
 \beta^{-1}(C) \lra C \lra 0 $$ of the original
extension.  However there will in general be many other
subextensions with last term $C$.
In fact, the reader could check that the one just mentioned is just
the `universal completion' extension discussed in the section on
`Diagram I'.
Indeed, there may be many different subalgebras of ${\mathcal E}$
which give subextensions with first term $A$ and last term $C$.  So
$A$ and $C$ do not determine $B$ uniquely.

 We next
suppose that we are given an approximately unital closed subalgebra
$A$ of ${\mathcal D}$.   By the previous result, the subextensions
with first term $A$ are in bijective correspondence with the
subobjects $B$ of ${\mathcal E}$ such that $B \cap {\mathcal D} =
A$. If one would prefer to characterize subextensions starting with
$A$ in terms of $C$ rather than $B$ then the situation seems much
more complicated.  However, in the case that $A$ contains a cai for
${\mathcal D}$, there is such a characterization of subextensions
which is fairly straightforward. We remark that this `common cai'
condition is probably the most interesting case of subextensions.
For example, very often ${\mathcal D}$ is a $C^*$-algebra generated
by $A$, so that by \cite[Lemma 2.1.7]{BLM} they share a common cai.

\begin{theorem} \label{new}  Given an extension
$$0 \to {\mathcal
 D} \overset{\alpha}\lra {\mathcal E}
 \overset{\beta}\lra {\mathcal F} \to 0 ,$$
and a closed subalgebra $A$ of ${\mathcal D}$ which contains a cai
for ${\mathcal D}$, then the subextensions beginning with $A$ are in
bijective correspondence with the
nontrivial subobjects $C$ of
$$ {\mathcal G} \overset{def}{=} \{ c \in {\mathcal F} : \exists
b \in {\mathcal E} \; \text{with} \; b A + A b \subset A \;
\text{such that} \; \beta(b) = c \} .$$ The middle term in the
ensuing subextension is unique and given by the formula
$$B = \{ b \in \beta^{-1}(C) : b A + A b
\subset A \} .$$ Moreover, the Busby invariant $\tau'$ for the
subextension is related to the Busby invariant $\tau$ for the
original extension by the formula $j(\tau'(c)) = \tau(c)$ for any $c
\in C$, or equivalently $\tau' = j^{-1} \circ \tau_{\vert C}$, where
$j$ is the canonical completely isometric morphism $\co(A) \to
\co({\mathcal D})$ from Corollary {\rm \ref{tilde}}.
\end{theorem}

\begin{proof}
Given any subextension with terms $A, B, C$, then $C$ is clearly a
closed   subalgebra of ${\mathcal G}$, and $B \subset \{ b \in
\beta^{-1}(C) : b A + A b \subset A \}$. Conversely, if $b \in
\beta^{-1}(C)$ with $b A + A b \subset A$, then there exists $b_1
\in B$ with $b - b_1 \in {\rm Ker}(\beta)$, so that $b - b_1 \in
{\mathcal D}$.   If $d = b - b_1$ then $d A + A d \subset A$.
Thus
 $d = \lim e_t d \in A$, where $(e_t)$ is the common cai,
 so that $b = d + b_1 \in B$.

 Conversely, given a
subobject $C$ of ${\mathcal G}$, we will show that $A, B, C$
constitute a subextension.  Clearly $B$ is closed. Suppose that $b
\in {\mathcal D}$ satisfies $b A + A b \subset A$. Then $b e_t \in
A$, where $(e_t)$ is a common cai for $A$ and ${\mathcal D}$, so
that $b \in A$.  This shows that $A, B, C$ constitutes an exact
sequence, and by Lemma \ref{try}
 we have a subextension.

To see the last assertion, let $c \in C, b \in B, \beta(b) = c$.
Then $\tau'(c) = \sigma'(b) + A, \tau(c) = \sigma(b) + {\mathcal D}$,
 and so the result boils down to showing that the
 canonical embedding $\iota: \M(A) \subset \M({\mathcal D})$ takes
 $\sigma'(b)$ to $\sigma(b)$.  Here $\sigma, \sigma'$ are
 the canonical maps from $B, {\mathcal E}$ into
 $\M(A), \M({\mathcal D})$ respectively.
 However for $b \in B, d \in {\mathcal D}$,
 $$\iota(\sigma'(b))(d) = \lim_t \, \sigma'(b)(e_t) d
 = \lim_t \, b e_t d = b d,$$
 and $\sigma(b)(d) = bd$.
 \end{proof}

{\bf Remarks.}  1)\ The set ${\mathcal G}$ above is a subalgebra of
${\mathcal F}$.  If we are working in the category {\bf OA} then
clearly there is a largest subextension with first term $A$, namely
$$0 \to A \overset{\alpha}\lra \{ b \in {\mathcal E} : b A + A b
\subset A \}
 \overset{\beta}\lra {\mathcal G} \to 0 .$$
 If we are working in the category {\bf AUOA} then
the same is true if ${\mathcal G}$ above has a cai, which happens
for example if ${\mathcal E}$ is unital.

2)\ If $A = {\mathcal D}$ then the above all follows from `Diagram
I'.

\bigskip

If $A$ is a $C^*$-algebra, and if we have an extension
$$0 \to A \overset{\alpha}{\lra} B \overset{\beta}{\lra} C \to 0 ,$$
then
there is a canonical $C^*$-algebra subextension.  Note that by
2.1.2 in \cite{BLM} the diagonal $\Delta(B) = B \cap B^*$ contains
$\alpha(A)$. If the diagonal $\Delta(C) = C \cap C^*$  is a
nontrivial subobject of $C$, then by Theorem \ref{new} we get a
subextension with first term $A$, last term $\Delta(C)$, and middle
term $\beta^{-1}(C)$.  The latter contains $\Delta(B)$ clearly (by
2.1.2 in \cite{BLM} again), and hence equals $\Delta(B)$ by the
remark after Proposition \ref{aau}.  Thus we have a $C^*$-algebra
subextension
$$0 \to A \overset{\alpha}{\lra} \Delta(B)
\overset{\beta_{\vert \Delta(B)}}{\lra} \Delta(C)  \to 0 .$$

\medskip

{\bf Examples.} 1. \  Every extension of the upper triangular matrix
algebra $T_n$ by $\Kdb$ is split. To see this, suppose that we have
an extension
$$0 \to A \overset{\alpha}{\lra} B \overset{\beta}{\lra} T_n \to 0 ,$$
where $A$ is a $C^*$-algebra.  It follows that $\alpha(A) \subset
\Delta(B) = B \cap B^*$.  By the remark after Theorem \ref{new} we
get a subextension
$$0 \to A \overset{\alpha}{\lra} \Delta(B) \overset{\beta}{\lra} D_n \to 0,$$
where  $D_n = \Delta(T_n)$.  If $A = \Kdb$, then this is just an
extension of $\ell^\infty_n$ by $\Kdb$, and we can lift the $n$
minimal projections in $D_n$ to $n$ mutually orthogonal projections
$p_i$ in $\Delta(B)$ (see e.g.\ \cite{BDF}).  Pick contractions $T_i
\in B$ with $\beta(T_i) = e_{i,i+1}$.  By replacing $T_i$ by $p_i
T_i p_{i+1}$, we may assume that $T_i = p_i T_i p_{i+1}$.  Let
$b_{ij} = T_i T_{i+1} \cdots T_{j-1}$, for $i < j$, $b_{ii} = p_i$.
The map $[\alpha_{ij}] \in T_n \mapsto \sum_{ij} \, \alpha_{ij}
b_{ij} \in B$ is a completely contractive homomorphism, by a result
of McAsey and Muhly (see \cite{RONC}), and is clearly a splitting
for the extension.

\medskip

2)\  Every unital extension of the disk algebra, or more generally
of Popescu's noncommutative disk algebra $A_n$
 (see \cite{Pop}), has a strongly unital splitting.  This follows
 by the associated noncommutative von Neumann
inequality.  For example,
unital morphisms on the disk algebra are in bijective
correspondence with contractions in the algebra that the
morphism maps into; and if the latter algebra is a
quotient algebra then we can lift the contraction and hence
 can lift the morphism too.  By a similar argument,
every nonunital extension of $A_n$ by $\Kdb$ splits.  For the
bidisk algebra $A(\Ddb^2)$, we can follow the obvious argument
for the disk algebra to see that the splitting of unital extensions
of $A(\Ddb^2)$ by the compacts say,
amounts to lifting commuting pairs of contractions in $\Bdb/\Kdb$
to commuting pairs of contractions in $\Bdb$, and
Ando's theorem for such pairs (see e.g.\ 2.4.13 in \cite{BLM}).
It is known that some such pairs do lift, while others
do not (see e.g.\ \cite{BO}), and so 
there are quite nontrivial extensions in this case.
In the sequel paper we will see that $Ext$ in this simple case
already brings up interesting operator theoretic topics.
However for the tridisk algebra $A(\Ddb^3)$,
and for algebras of analytic functions on
other classical domains, the argument above based on von Neumann
inequalities fails, although it is clear that one will
usually get non-split unital extensions.

\medskip

We will not treat the subject of {\em corona extendibility}
\cite{ELP} here, but will give a very simple, but very common,
example of it. Namely, given an extension $$0 \to A \lra B \lra C
\to 0$$ in the category {\bf OA}, suppose that $C$ is nonunital, and
$C^1$ is the Meyer unitization of $C$.  By Meyer's theorem
\cite{Mey,BLM}, the Busby invariant $\tau$ extends to a unital
morphism
 $\tau^1 : C^1 \to \co(A)$.     We leave it as an exercise that
 this gives the `superextension':
$$ 0 \lra A \lra B^1 \lra C^1 \lra 0 .$$
    We call this the {\em unitization extension}.
    The original extension is a subextension of this one.
 Conversely, any unital extension of $C^1$ by $A$ is
    the unitization extension of an extension of $C$ by $A$.  We
    leave this as an exercise in diagram chasing, using
    Theorem \ref{bus} and Meyer's theorem.

\section{Covering extensions}  In this section we start with an
extension $$E \, : \; \; \; 0 \lra A\overset{\alpha}\lra
B\overset{\beta}\lra C \lra 0 ,$$ and construct another extension
containing $E$ as a subextension.
Note that given any operator algebra $A'$ containing $A$ with a
common cai, by Diagram III we obtain a smallest (or universal)
`superextension' with first term $A'$, namely
$$
\xymatrix{
    0 \ar[r] & A   \ar[r]^{\alpha}  \ar@{^{(}->}[d] &  B
    \ar[d]  \ar[r]^{\beta} &
    C \ar[r]  \ar@{=}[d] & 0 \\
 0  \ar[r] &     A'  \ar[r]  & B'
    \ar[r]  &
    C \ar[r] & 0\\
    } $$
By Lemma \ref{five}, the middle arrow is a complete isometry, so
that the new extension contains $E$ as a subextension.

We are, however, more interested in superextensions consisting of
$C^*$-algebras. If each of these $C^*$-algebras is a $C^*$-cover of
the operator algebra in the matching place in the sequence $E$, then
we call the $C^*$-algebra extension a {\em covering extension} of
$E$. We now discuss these. For simplicity of exposition we will
occasionally assume that $A \subset B$ and $C = B/A$.

If we have a covering extension
$$
\xymatrix{
    0 \ar[r] & A   \ar[r] \ar[d]_\mu &  B
    \ar[d]^j  \ar[r] &
    C \ar[r]  \ar[d]^\nu & 0 \\
 0  \ar[r] &   {\mathcal D}    \ar[r]  & {\mathcal E}
    \ar[r]  &
    {\mathcal F} \ar[r] & 0\\
    } $$
where $\mu, j, \nu$ are the canonical maps from $A, B, C$
respectively into the given $C^*$-covers, then the Busby invariant
$\tau^*$ of the covering extension is related to the Busby invariant
$\tau$ of the original extension by the formula
\begin{equation} \label{ones}  \tilde{\mu} \circ
\tau = \tau^* \circ \nu = \tau^*_{\vert C},\end{equation}
  by
Corollary \ref{bumo}, where we are viewing $C \subset {\mathcal F}$
via $\nu$. Since $\tilde{\mu}$ is completely isometric by Corollary
\ref{tilde}, we have \begin{equation} \label{one} \tau =
\tilde{\mu}^{-1} \circ (\tau^*)_{\vert C}.\end{equation}

\medskip

{\bf Remark.}  An extension is completely essential
iff there is an essential covering extension with first two terms
$C^*$-envelopes, and iff there exists
some covering extension which is essential.  This follows from
Proposition  \ref{cesid} (and the proof of Lemma \ref{cenvi}).

\medskip

 As was the case for subextensions (see Proposition \ref{etc}), we have:

\begin{proposition}  \label{etc2}   Given an approximately unital
ideal in an
operator algebra $B$,  the equivalence classes (with respect to
strong isomorphism) of covering extensions of an extension
$$0 \to A
  \lra B \lra C \to 0 ,$$
 are in an  bijective
correspondence with the equivalence classes of $C^*$-covers
$({\mathcal E},j)$ of $B$.
 \end{proposition}

\begin{proof}  Note that such ${\mathcal E}$ gives a covering
extension: set ${\mathcal D} = C^*_{{\mathcal E}}(j(A))$, and set
${\mathcal F} = {\mathcal E}/{\mathcal D}$.   By Lemma \ref{id},
${\mathcal D}$ is a two-sided ideal in ${\mathcal E}$. There  is a
canonical completely isometric homomorphism $\nu : C = B/A \to F$
(see Lemma \ref{try2}). It is easy to see that $({\mathcal F},\nu)$
is a $C^*$-cover of $C$. With $\mu = j_{\vert A} : A \to {\mathcal
D}$ we have a commutative diagram
$$
\xymatrix{
    0 \ar[r] & A   \ar[r] \ar[d]_\mu &  B
    \ar[d]^j  \ar[r] &
    C \ar[r]  \ar[d]^\nu & 0 \\
 0  \ar[r] &   {\mathcal D}    \ar[r]  & {\mathcal E}
    \ar[r]  &
    {\mathcal F} \ar[r] & 0\\
    } $$
 with all vertical arrows completely isometric homomorphisms.
 \end{proof}

{\bf Remark.}  An interesting consequence of this, is that it
follows that the equivalence classes of covering extensions of a
given extension are in a bijective order reversing correspondence
with the open sets in a certain topology, by p.\ 99 of \cite{BLM}.

\begin{lemma}  \label{enf} In the definition above of a
covering extension, it is not necessary to assume that ${\mathcal
E}$ is a $C^*$-cover of $B$.  This is automatically implied by the
other hypotheses.
\end{lemma}

\begin{proof}
Let ${\mathcal G}$ be the $C^*$-subalgebra of
${\mathcal E}$ generated by $B$.  The image (resp.\ inverse image)
of $G$ in ${\mathcal F}$ (resp.\ ${\mathcal D}$) is a $C^*$-algebra,
and hence must equal ${\mathcal F}$ (resp.\ ${\mathcal D}$) since
the latter is generated by $C$  (resp.\ $A$). By the five lemma
 the inclusion map ${\mathcal G} \to {\mathcal E}$ is
surjective.  That is, ${\mathcal G} = {\mathcal E}$.
\end{proof}

\begin{lemma}  \label{en}  Given an extension
$$0 \lra A \lra B \lra C \lra 0$$
as above, consider the two canonical maps from the proof of
Proposition {\rm \ref{etc2}}, from the set of $C^*$-covers of $B$ to
the set of $C^*$-covers of $A$, and to the set of $C^*$-covers of
$C$. These two maps
preserve the natural ordering of $C^*$-covers. The first of these
maps is surjective.
\end{lemma}

\begin{proof}
We leave the first assertion as an exercise in diagram chasing.  To
see the second, note that for any $C^*$-cover ${\mathcal D}$ of $A$,
we have $\co(A) \subset \co({\mathcal D})$ completely isometrically
via the map $\tilde{\mu}$ above, by Corollary \ref{tilde}. The map
$\tilde{\mu} \circ \tau$ on $C$ extends to a $*$-homomorphism
$\tau^* : C^*_{\rm max}(C) \to \co({\mathcal D})$, which is the
Busby invariant for an extension of $C^*_{\rm max}(C)$ by ${\mathcal
D}$. Since $\tilde{\mu} \circ \tau = \tau^*_{\vert C}$, we see by
Theorem \ref{bumo} that we have a morphism of extensions, the middle
vertical arrow being completely isometric by Lemma \ref{five}. By
Lemma \ref{enf} we have constructed a covering extension with first
term ${\mathcal D}$. \end{proof}

{\bf Remarks.}  Neither of the two maps in the last lemma are
one-to-one in general.  The second map need not be onto, even if the
extension is completely essential.
 An example showing this may easily be constructed in the
case that $A = \Kdb$ and $C$ is the upper triangular $2 \times 2$
matrices. In this case one may easily construct (as in the next
result) an
 extension with Busby invariant $\tau : C \to \Bdb/\Kdb$, such that
$\tau$ has no extension to a $*$-homomorphism from $C^*_{\rm e}(A)$
into $\Bdb/\Kdb$.  We can even ensure that the extension is
completely essential and trivial.  For example, choose in $\Bdb$ a
projection $p$, and a partial isometry $u$ with $u = p u p^{\perp}$,
but $u u^* - p \notin \Kdb$ and $u^* u - p^{\perp} \notin \Kdb$.
Let $\dot{p}, \dot{u}$ be the corresponding elements of $\Bdb/\Kdb$.
Then it is easy to see that using e.g.\ \cite[Corollary
2.2.12]{BLM}, that the
map $$\tau : C \to \Bdb/\Kdb : \left[ \begin{array}{ccl} a & b \\
0 & c \end{array} \right] \mapsto a \dot{p} + b \dot{u} + c (1-
\dot{p}) ,$$ corresponds to a completely essential trivial
extension, but does not extend to a $*$-homomorphism on $M_n$. Hence
there is no covering extension of this extension with third term
$C^*_{\rm e}(A)$. Note this is also an example in which $C^*_{\rm
e}(B)/C^*_{\rm e}(A)$ is not isomorphic to $C^*_{\rm e}(B/A)$.

\begin{proposition}  \label{etc3}   Given an extension
$$E : 0 \to A
  \lra B \lra C \to 0 ,$$ as above, and
$C^*$-covers $({\mathcal D},\mu)$ and $({\mathcal F},\nu)$ of $A$
and $C$ respectively, then up to strong isomorphism there exists at
most one covering extension of $E$ with first and third terms
${\mathcal D}$ and ${\mathcal F}$.  In fact there will exist one of
these if and only if the canonical map $\tilde{\mu} \circ \tau \circ
\nu^{-1} : \nu(C) \to \co({\mathcal D})$ extends to a
$*$-homomorphism
 ${\mathcal F} \to \co({\mathcal D})$.
 \end{proposition}

 \begin{proof}
 This essentially follows from Equation (\ref{ones}). If
$\tilde{\mu} \circ \tau \circ \nu^{-1}$ extends to a
$*$-homomorphism $\tau^* : {\mathcal F} \to \co({\mathcal D})$, then
the latter is the Busby invariant of a $C^*$-algebraic
 extension of
${\mathcal F}$ by ${\mathcal D}$.  By Theorem \ref{bumo} and Lemmas
\ref{five} and \ref{enf}, it is easy to see that this is a covering
extension.

 Because $\nu(C)$ generates ${\mathcal F}$,
$*$-homomorphisms on ${\mathcal F}$ are determined uniquely by their
restrictions to $\nu(C)$.  Thus there is at most one
$*$-homomorphism $\tau^* : {\mathcal F} \to \co({\mathcal D})$
satisfying $\tilde{\mu} \circ \tau = \tau^*_{\vert C}$.
\end{proof}

{\bf Remark.}  There may exist no covering extension of the type
mentioned in the last result, as is pointed out in the last Remark.

\begin{proposition}  \label{etc5}   For any separable
operator algebra $C$, and stable approximately unital operator
algebra $A$, there exists an essential split  extension of $C$ by
$A$. The middle term in this extension may be chosen to be nonunital
if we wish, or to be unital if $C$ is unital.
 \end{proposition}

 \begin{proof}   Let $\pi : C^*_{\rm max}(C) \to B(\ell^2)$ be
 a faithful $*$-representation: this is possible since
 $C^*_{\rm max}(C)$ is separable.  Furthermore, we can assume that
 $\pi(C^*_{\rm max}(C)) \cap \Kdb = (0)$, by replacing $\pi$ by
 $\pi \oplus \pi \oplus \cdots$ (unital case) or by
$0 \oplus \pi \oplus \pi \oplus \cdots$ (nonunital case).   Now
 $$B(\ell^2) = \M(\Kdb) \subset \M(A \otimes \Kdb)
 \cong \M(A) ,$$
 and so we obtain a faithful completely contractive representation
 $\theta : C^*_{\rm max}(C) \to \M(A)$ for which it is easy to see
 that
 $\theta(C^*_{\rm max}(C)) \cap A = (0)$.  Consider the maps
 $$C^*_{\rm max}(C) \to \M(A) \to \co(A) \to \co(C^*_{\rm max}(A))
 .$$
 These compose to a $*$-homomorphism, by \cite[Proposition 1.2.4]{BLM}.  Note that if $\theta(x)
 \in C^*_{\rm max}(A)$, then $$\theta(x)
 \in \M(A) \cap C^*_{\rm max}(A) \subset A^{\perp \perp} \cap C^*_{\rm max}(A)
 = A,$$
by \cite[Lemma A.2.3 (4)]{BLM}, so that $\theta(x) = 0$.  Thus the
composition of the
 maps in the last centered sequence is faithful, and so
 completely isometric.  Hence the associated morphisms $C^*_{\rm max}(C) \to \co(A)$ and
 $C \to \co(A)$ are completely isometric, and they factor through
 $\M(A)$.   \end{proof}

\begin{corollary}  \label{finise}  If we have a completely
essential
extension
$$E: \; \; \; 0 \to A
  \lra B \lra C \to 0 ,$$
  and a
$C^*$-cover $({\mathcal D},\mu)$ of $A$, then there exists  a
`smallest' or universal covering extension
$$E_{\rm min} \, : \; \; 0 \lra {\mathcal D} \lra {\mathcal E} \lra {\mathcal F} \lra 0 $$
 with first term ${\mathcal D}$.   More
particularly, $E_{\rm min}$ is a quotient extension of any other
covering extension of $E$ with first term ${\mathcal D}$. Also,
$E_{\rm min}$ is essential.
 \end{corollary}

\begin{proof}
 We have a complete isometry
$\kappa = \tilde{\mu} \circ \tau : C \to \co(A) \to \co({\mathcal
D})$. The $C^*$-algebra ${\mathcal F}$ generated by $\kappa(C)$ is a
$C^*$-cover of $C$.  By Proposition \ref{etc3},   the inclusion map
${\mathcal F} \to \co({\mathcal D})$ is the Busby invariant of a
covering extension $E_{\rm min}$ with first and third terms
${\mathcal D}$ and ${\mathcal F}$.   Given any other covering
extension of $E$ with first term ${\mathcal D}$  and last term
${\mathcal G}$ say, there is a $*$-homomorphism $\pi : {\mathcal G}
\to \co({\mathcal D})$ such that $\pi(C) = \tilde{\mu}(\tau(C))
\subset {\mathcal F}$. It follows that $\pi : {\mathcal G} \to
{\mathcal F}$. The conditions of Theorem
\ref{bumo} are met, so that $E_{\rm min}$ is a quotient
extension of the extension of ${\mathcal G}$.
\end{proof}

{\bf Remark.}  As in \cite{ELP}, the above universal covering
extension may be described as a pushout.

\begin{proposition}  \label{exce}  With notation as in
Proposition  {\rm \ref{etc3}}, if ${\mathcal D}$ and ${\mathcal F}$
are $C^*$-envelopes (resp.\ maximal $C^*$-covers), and if a covering
extension of $E$ does exist with ${\mathcal D}$ and ${\mathcal F}$
as first and third terms,
then the middle term ${\mathcal E}$ is also a $C^*$-envelope (resp.\
a maximal $C^*$-cover).
\end{proposition}

 \begin{proof}
 To see
the first claim, note that if this middle term is ${\mathcal E}$,
which dominates $C^*_{\rm e}(B)$, then by the fact in Lemma \ref{en}
about the two maps being order preserving, we see that there exists
a covering extension with middle term $C^*_{\rm e}(B)$, and other
two terms dominated by, and hence equal to, $C^*_{\rm e}(A)$ and
$C^*_{\rm e}(C)$ respectively. By the remark at the start of the
paragraph, ${\mathcal E} = C^*_{\rm e}(B)$.  We leave the second as
an exercise, using a similar idea.
\end{proof}

A covering extension of the type in the last Proposition with all
terms $C^*$-envelopes, if it exists, will be called a {\em
$C^*$-enveloping extension}, and the extension itself (of $C$ by
$A$) will be called a {\em $C^*$-enveloped extension}.   If the
Busby invariant $\tau$ of an extension of $C$ by a $C^*$-algebra $A$
 extends to a $*$-representation
$C^*_{\rm e}(C) \to \co(A)$,
 then the extension $E$ is
$C^*$-enveloped.  In particular, any nonselfadjoint operator algebra
extension of a $C^*$-algebra is $C^*$-enveloped.

\medskip

{\bf Example.}  There are many very interesting and topical examples
of $C^*$-enveloped extensions, for example coming from the
generalizations of Gelu Popescu's noncommutative disk algebra $A_n$
which have attracted much interest lately.  The way in which these
are usually obtained is one finds a `Toeplitz-like' $C^*$-algebra
${\mathcal E}$ with a quotient `Cuntz-like' $C^*$-algebra ${\mathcal
F}$, which in turn is generated by a nonselfadjoint operator algebra
$A$.  In Popescu's original setting \cite{Pop} the picture is:
$$
\xymatrix{
    0 \ar[r] &  \Kdb  \ar[r] & C^*(S_1, \cdots, S_n)
    \ar[r] &
{\mathcal O}_n   \ar[r] &  0\\
    0  \ar[r] &   \Kdb  \ar@{=}[u] \ar[r]
    &          \circ
      \ar[r] \ar@{^{(}->}[u] &
    A_n \ar[r] \ar@{^{(}->}[u] & 0 \\
    } $$
  (When $n = 1$,
  $A_n$ is just the disk algebra,
  and the top row  is just the
  Toeplitz extension by the compacts.)
   In any such setting, by our earlier theory (e.g.\ Theorem
    \ref{new}),
    there is a unique completion of the diagram to a subextension.
    Indeed, in the example above, the missing term in the diagram is
    the inverse image
    under the top right arrow
    $\beta : C^*(S_1, \cdots, S_n) \to {\mathcal O}_n$ of the bottom right algebra
$A_n$, which is the closure in $C^*(S_1, \cdots, S_n)$ of $\Kdb +
A_n$.
If one can show that the top
right $C^*$-algebra (${\mathcal F}$ in language above) is a
$C^*$-envelope of the bottom right algebra, and doing this is
currently quite an industry (initiated by Muhly and Solel, see e.g.\
\cite{MS,KK}), then it follows from Proposition \ref{exce}, that the
covering extension is $C^*$-enveloping.

\medskip

 {\bf Remark.}  We give
another proof of Lemma \ref{cenvi}: If $A$ is a closed approximately
unital ideal in an operator algebra $B$, then by the proof of Lemma
\ref{en} we have a covering extension
$$0 \lra C^*_e(A) \to {\mathcal E} \lra C^*_{\rm max}(C) \lra 0 .$$
On the other hand, by Proposition  \ref{etc2}, there is another
covering extension with middle term $C^*_e(B)$.  Since the latter is
dominated by ${\mathcal E}$ in the ordering of $C^*$-covers, it
follows from Lemma \ref{en} that the first term in the last covering
extension is dominated by, and hence equals, $C^*_e(A)$. Thus
 $C^*_e(A)$ is an ideal in $C^*_e(B)$.

Similar reasoning gives another proof of Lemma \ref{max}. That is:

\begin{lemma}  \label{max2}  If the middle $C^*$-algebra in
a  covering extension  is a
 maximal $C^*$-cover, then all the $C^*$-algebras in the covering extension
 are maximal $C^*$-covers.
 \end{lemma}

 If $A$ is an approximately unital operator algebra, then
 since $\co(A) \subset \co(C^*_{\rm max}(A))$ by Corollary
\ref{tilde}, any morphism $C \to \co(A)$ extends uniquely to a
$*$-homomorphism $C^*_{\rm max}(C) \to  \co(C^*_{\rm max}(A))$.
  By the Busby correspondence, we see that
  this defines a one-to-one map from {\bf Ext}$(C,A)$
   into {\bf Ext}$(C^*_{\rm max}(C),C^*_{\rm
max}(A))$.   This map is not in general surjective.  However it will
be if $A$ is a $C^*$-algebra:

\begin{corollary} \label{maxj}  Let ${\mathcal D}$ be a $C^*$-algebra.
 There is a canonical bijection from {\bf Ext}$(C,{\mathcal D})$
onto {\bf Ext}$(C^*_{\rm max}(C),{\mathcal D})$, taking the
 split extensions onto the split extensions.
\end{corollary}

\begin{proof}
The first assertion is clear from the
discussion above, since any  morphism $C^*_{\rm max}(C) \to
  \co({\mathcal D})$ uniquely extends a  morphism $C \to
  \co({\mathcal D})$, by the universal property of $C^*_{\rm max}$.
We  leave the last assertion to the reader (see also \cite{Mth}).
 \end{proof}

In particular, {\bf Ext}$(C) = {\rm {\bf Ext}}(C^*_{\rm max}(C))$,
where as usual {\bf Ext}$(\cdot)$ means extensions by $\Kdb$.
We  study the associated semigroup/group in the sequel paper,
which turns out to have very many of the important properties
that one has in the $C^*$-algebra case.    It is rarely a group; 
if one wants a group one can look at 
the invertible elements of {\bf Ext}$(C)$, or at the 
variant of $Ext$ corresponding to $C^*$-enveloping extensions.

\end{document}